\documentclass[12pt]{article}

\usepackage{amssymb,amsmath,amsthm}
\usepackage{epsfig}
\usepackage{breqn}
\usepackage{rotating}
\usepackage{amsmath}
\usepackage{amsfonts}
\usepackage{setspace}
\usepackage{fullpage}
\usepackage{enumitem}
\usepackage{bbold} 
\usepackage{comment}
\usepackage{tikz}
\usepackage{pgf}
\usetikzlibrary{positioning,arrows,shapes,decorations.markings,decorations.pathreplacing,matrix}
\tikzstyle{vertex}=[circle,draw=black,fill=black,inner sep=0,minimum size=3pt,text=white,font=\footnotesize]
\usepackage{mathtools}  
\usepackage{hyperref}
\hypersetup{
	colorlinks=true,
	linkcolor=blue,
	filecolor=magenta,
	urlcolor=cyan,
	citecolor=blue
}
\mathtoolsset{showonlyrefs} 
\usepackage{authblk}
\usepackage[none]{hyphenat}

\newcommand\cF{{\mathcal F}}

\newcommand{\abs}[1]{\left\lvert{#1}\right\rvert}
\newcommand{\floor}[1]{\left\lfloor{#1}\right\rfloor}

\theoremstyle{plain}
\newtheorem{theorem}{Theorem}[section]
\newtheorem*{rstthm1.1}{Theorem 1.1}
\newtheorem*{rstthm1.2}{Theorem 1.2}
\newtheorem*{rstthm1.3}{Theorem 1.3}
\newtheorem{lemma}[theorem]{Lemma}

\newtheorem{proposition}[theorem]{Proposition}

\newtheorem{claim}{Claim}

\DeclareMathOperator{\ex}{ex}

\title{Generalized rainbow Tur\'an problems}

\author{D\'aniel Gerbner\footnote{Alfr\'ed R\'enyi Institute of Mathematics, Hungarian Academy of Sciences, E-mail: \texttt{gerbner@renyi.hu.} Research supported by the J\'anos
    Bolyai Research Fellowship of the Hungarian Academy of Sciences and the
    National Research, Development and Innovation Office -- NKFIH under the
    grants K 116769, KH130371 and SNN 129364.}, \enskip Tam\'as M\'esz\'aros\footnote{Institut f\"ur Mathematik, Freie Universit\"at Berlin, E-mail: \texttt{tamas.meszaros@fu-berlin.de.} Research supported by the Dahlem Research School.}, \enskip Abhishek Methuku\footnote{Discrete Mathematics Group, Institute for Basic Science (IBS), Daejeon, Republic of Korea. E-mail: \texttt{abhishekmethuku@gmail.com}. Research was supported by IBS-R029-C1.
}, \enskip Cory Palmer\footnote{Department of Mathematical Sciences,
University of Montana, Missoula, Montana 59812, USA. E-mail: \texttt{cory.palmer@umontana.edu.}}}

\date{}

\begin{document}

\maketitle

\begin{abstract}
Alon and Shikhelman  [{\it J.\ Comb.\ Theory, B.} {\bf 121} (2016)] initiated the systematic study of the following generalized Tur\'an problem: for fixed graphs $H$ and $F$ and an integer $n$, what is the maximum number of copies of $H$ in an $n$-vertex $F$-free graph? 

An edge-colored graph is called {\it rainbow} if all its edges have different colors. The {\it rainbow Tur\'an number} of $F$ is defined as the maximum number of edges in a properly edge-colored graph on $n$ vertices with no rainbow copy of $F$. The study of rainbow Tur\'an problems was initiated by Keevash, Mubayi, Sudakov and Verstra\"ete [{\it Comb.\ Probab.\ Comput.} {\bf 16} (2007)]. 

Motivated by the above problems, we study the following problem: What is the maximum number of copies of $F$ in a properly edge-colored graph on $n$ vertices without a rainbow copy of $F$? We establish several results, including when $F$ is a path, cycle or tree.

\end{abstract}

\section{Introduction}

This paper is concerned with variations of the Tur\'an question in extremal graph theory. In the classic setting, given a simple graph $F$, we are interested in determining the largest possible number of edges in a simple graph $G$ on $n$ vertices without $F$ as a subgraph. This number is called the \emph{Tur\'an number} of $F$ and is denoted by $\ex(n,F)$. In short we will say $G$ is \emph{$F$-free}. The prototypical result in the area is Mantel's theorem from 1907 \cite{Mant07}. Mantel showed that if an $n$-vertex graph does not contain a triangle, it can have at most $\lfloor\tfrac{n^2}{4}\rfloor$ edges, and this bound is best possible as shown by the balanced complete bipartite graph. Therefore, we have  $\ex(n,K_3)=\lfloor\tfrac{n^2}{4}\rfloor$. Tur\'an \cite{T1941} generalized this to all cliques, and determined $\ex(n,K_k)$ for every $k$ and $n$. A  general result was proven by Erd\H{o}s and Simonovits \cite{ESi} as a corollary to a theorem of Erd\H{o}s and Stone \cite{ESt}. They proved that for any simple graph $F$ we have $\ex(n,F) = \left( 1-\tfrac{1}{\chi(F)-1}\right) \binom n2+o(n^2)$, where $\chi(F)$ is the chromatic number of $F$. If $F$ is not bipartite, this theorem determines $\ex(n,F)$ asymptotically. However, for bipartite graphs the Erd\H{o}s-Stone-Simonovits theorem just states that $\ex(n,F)$ is of lower than quadratic order. A general classification of the order of magnitude of bipartite Tur\'an numbers is not known. For paths Erd\H{o}s and Gallai \cite{EG59} showed that $\ex(n,P_k)\leq \tfrac{1}{2}(k-2)n$, where $P_k$ denotes the path on $k$ vertices and equality holds for the graph of disjoint copies of $K_{k-1}$. Erd\H{o}s and S\'os \cite{Er46} conjectured that the same should  hold  for any fixed tree on $k$ vertices. A proof of this conjecture for $k$ large enough was announced by Ajtai, Koml\'os, Simonovits and Szemer\'edi. For even cycles Erd\H{o}s conjectured that $\ex(n,C_{2k})=\Theta(n^{1+{1}/{k}})$. A corresponding upper bound was given by Bondy and Simonovits \cite{BS1974}, but so far a matching lower bound has only been found for $k=2,3,5$ (\cite{B1966,W1991}). For more results the interested reader may consult the comprehensive survey on bipartite Tur\'an problems by F\"uredi and Simonovits \cite{fs}.

\smallskip

The classical Tur\'an problem has a rich history in combinatorics and several variations and generalizations of it have been studied. Two such variations are rainbow Tur\'an problems and generalized Tur\'an problems. In this paper we will study a natural generalization of these two problems.

\smallskip

The rainbow Tur\'an problem, introduced by Keevash, Mubayi, Sudakov, and Verstra{\"e}te \cite{kmsv}, is as follows. For a fixed graph $F$, determine the maximum number of edges in a properly edge-colored graph on $n$ vertices which does not contain a {\it rainbow} $F$, i.e., a copy of $F$ all of whose edges have different colors. This maximum is denoted by $\ex^*(n,F)$ and is called the \emph{rainbow Tur\'an number} of $F$. (We refer the reader to \cite{kmsv} for a discussion on motivations and applications of this problem.) Observe that by definition we always have $\ex^*(n,F)\ge \ex(n,F)$. In relation with the Erd\H{o}s-Stone-Simonovits theorem in \cite{kmsv} it was shown that if $\chi(F)\ge 3$, then $\ex^*(n,F)=(1+o(1))\ex(n,F)$. In the case of the path $P_k$ on $k$ vertices we know that $\frac{k-1}{2}n\leq \ex^*(n,P_k)\leq \left(\frac{9k-4}{7}\right)n$, where the lower bound is due to Johnston and Rombach \cite{jr}, while the upper bound was proven by Ergemlidze, Gy\H{o}ri and Methuku \cite{EGyM19}. For general trees only some sporadic results are known, for such results see. e.g. \cite{jps,jr}. In the case of even cycles Keevash, Mubayi, Sudakov and Verstra{\"e}te showed a general lower bound of $\ex^*(n,C_{2k})=\Omega(n^{1+1/k})$ and that there exists a graph with $\Omega(n\log n)$ edges without a rainbow cycle of any length. For $k=3$ they  also gave a matching upper bound and hence showed that $\ex^*(n,C_6)=\Theta(n^{\frac{4}{3}})$. On the other hand, they also showed that asymptotically $\ex^*(n,C_6)$ is a constant factor larger then $\ex(n,C_6)$. The best known general upper bound on the rainbow Tur\'an number of even cycles is due to Das, Lee and Sudakov \cite{DLS13}, who showed that $\ex^*(n,C_{2k})=(1+o(1))O\left(n^{1+\frac{(1+\epsilon_k)\ln k}{k}}\right)$, where $\epsilon_k\rightarrow 0$ as $k\rightarrow \infty$.

\smallskip

The generalized Tur\'an problem asks the following. Given two graphs $H$ and $F$, what is the maximum possible number of copies of $H$ in a graph on $n$ vertices without containing a copy of $F$? This maximum is called the \emph{generalized Tur\'an number} and is denoted by $\ex(n,H,F)$. Note that if $H=K_2$ (an edge), we recover the classical Tur\'an problem.
The first results concerning this function are due to Zykov \cite{zykov}, Erd\H{os} \cite{E62} and Bollob\'as \cite{B1976} who determined $\ex(n,K_r,K_k)$ for every value of $n$, $r$ and $k$. Later Bollob\'as and Gy\H{o}ri  \cite{BG2008} proved that $\ex(n,C_3,C_5)=\Theta\left({n^{{3}/{2}}}\right)$. Gy\H{o}ri and Li \cite{GL2012} gave bounds on $\ex(n,C_3,C_{2\ell+1})$. Another well-known result is due to Grzesik \cite{G2012} and independently to Hatami, Hladk\'y, Kr\'al', Norine and Razborov \cite{HHKNR2013}, who determined $\ex(n,C_5,C_3)$ exactly. Recently, the systematic study of $\ex(n,H,F)$ was initiated by Alon and Shikhelman \cite{ALS2016}, and this problem has 
attracted the interest of a number of researchers; see e.g., \cite{AKS2018, FKL2018, GGMV2017, GMV2017, GS2018, L2018, MQ, MYZ2018}.

\smallskip

Here we consider a natural generalization of the above two problems and introduce a new variant of the Tur\'an problem. Given two graphs $H$ and $F$, let $\ex(n,H,\textup{rainbow-}F)$ denote the maximum possible number of copies of $H$ in a properly edge-colored graph on $n$ vertices without containing a rainbow copy of $F$. Observe that, by definition, we always have $\ex(n,H,\textup{rainbow-}F)\ge \ex(n,H,F)$. In this paper we focus on the case when $H = F$. In other words, we consider the question: How many copies of $F$ can we have in a properly edge-colored graph on $n$ vertices without having a rainbow copy of $F$? 
Our motivation in studying this function comes from attempts to understand the original rainbow Tur\'an problem. To determine $\ex^*(n,F)$ it is important to separate the problem from the classical problem of determining $\ex(n,F)$. In order to do this, one needs to examine properly edge-colored $n$-vertex graphs that contain more than $\ex(n,F)$ edges without a rainbow copy of $F$; such graphs contain many copies of $F$ but no rainbow copy of $F$. Therefore, it is natural to understand how many copies of $F$ can we take before we are forced to have a rainbow copy of $F$. 

\smallskip

Let us introduce some basic notation that is used throughout this paper. For positive integers $p$, $k$ and $l$, let $P_k$ denote a path on $k$ vertices, let $C_l$ denote a cycle on $l$ vertices.  A \textit{star} is a tree which consists of a vertex that is adjacent to all the other vertices. Let $S_p$ denote the star with $p$ leaves (i.e., a vertex adjacent to $p$ other vertices). A tree is called a \textit{double star} if its longest path has four vertices, or equivalently, if there is an edge $uv$ such that every vertex is adjacent to $u$ or $v$. If $u$ is adjacent to $p$ leaves and $v$ is adjacent to $r$ leaves, we denote this graph by $S_{p,r}$ and we call both $u$ and $v$ the \emph{centers}. Note that $S_{1,1}$ is just the path $P_{4}$. 

\subsubsection*{Main results}

First we determine the order of magnitude of the function $\ex(n,P_k,\textup{rainbow-}P_k)$ for all $k$.

\begin{theorem}\label{path} If $k\ge 5$, then
\begin{equation*}
\ex(n,P_k,\textup{rainbow-}P_k) = \Theta(n^{\floor{\frac{k}{2}}}).    
\end{equation*}
\end{theorem}

For $k\in \{2,3\}$ note that we have $\ex(n,P_2,\textup{rainbow-}P_2)=\ex(n,P_3,\textup{rainbow-}P_3)=0$ and for $k=4$ we will show in Proposition \ref{p4} that $\ex(n,P_4,\textup{rainbow-}P_4)=\Theta(n)$.

\smallskip

Our next result is for cycles.


\begin{theorem}\label{cycles}  If $k,\ell\ge 2$, then
\begin{equation*}
\ex(n,C_{2k+1},\textup{rainbow-}C_{2k+1})=\Theta(n^{2k-1})
\end{equation*}
and
\begin{equation*}
\Omega(n^{k-1})=\ex(n,C_{2k},\textup{rainbow-}C_{2k})=O(n^k).
\end{equation*}
Moreover, if $k\neq \ell$, then 
\begin{equation*}
\ex(n,C_{2\ell},\textup{rainbow-}C_{2k})=\Theta(n^\ell).
\end{equation*}
\end{theorem}

\smallskip

Our final result is about trees. For a tree $T$, any rainbow-$T$-free graph $G$ can have at most a linear number of edges: Indeed a graph with sufficiently large (but constant) minimum degree contains a rainbow copy of $T$. Therefore, by a theorem proved in \cite{A1981}, the maximum possible number of copies of $T$ in $G$ is at most $O(n^{\alpha(T)})$, where $\alpha(T)$ is the size of a maximum independent set in $T$. This proves that $\ex(n,T,\textup{rainbow-}T)=O(n^{\alpha(T)})$ for any tree $T$. We show that $\ex(n,T,\textup{rainbow-}T)$ grows fast with the number of vertices of $T$.

\begin{theorem}\label{trees}
If $T$ is a tree with $t$ vertices that is neither a star nor a double star, then
\begin{equation*}
    \ex(n,T,\textup{rainbow-}T)=\Omega(n^{\left\lceil t/4\right \rceil}).
\end{equation*}
\end{theorem}

In connection with Theorem \ref{trees}, note that if $T$ is a star then clearly $\ex(n,T,\textup{rainbow-}T)=0$, as every properly edge-colored star is necessarily rainbow. To complete the picture, we will prove in Proposition \ref{doublestar} that for double stars the answer is $\Theta(n)$.

\smallskip

\textbf{Outline of the paper.}  This paper is organized as follows. In Section \ref{sec:genBounds} we prove some general bounds on the function $\ex(n,F,\textup{rainbow-}F)$. Then, in Section \ref{sec:PathsCycles} we prove Theorem~\ref{path} and Theorem~\ref{cycles}. In Section \ref{sec:TreesForests}, we prove Theorem~\ref{trees}, together with corresponding results about forests of special types, and we make some concluding remarks and present open problems in Section \ref{sec:ConcludingOpen}.

\smallskip

\textbf{Notation.}  In the proofs we will often use the following operation to construct extremal examples. Given a graph $G$ and a vertex $v$, we delete $v$ from $G$ and replace it by new vertices $v_1,\dots,v_b$, each of them connected to the neighbors of $v$. To refer to this operation we will simply say that we {\it replace} $v$ by $b$ copies of itself. Usually, we will mostly be interested in the order of magnitude of the function $\ex(n,H,\textup{rainbow-}H)$. So when applying this operation (mostly to some small graph $G$), we will not specify the exact value of $b$, but we will only write $b=\Theta(n)$. By this we will always mean that we choose $b$ to be $cn$ for some appropriate constant $c$, such that the resulting graph still has at most $n$ vertices. 

\section{General bounds}
\label{sec:genBounds}

Using the graph removal lemma, we show that for any graph $H$, the number of copies of $H$ in a rainbow-$H$-free graph is at most $o(n^{\abs{V(H)}-1})$. More precisely:

\begin{proposition}
\label{GenBounds}
 For any graph $H$ on $k$ vertices, we have
\begin{itemize}
    \item[(i)] $\ex(n, H, \textup{rainbow-}H) =O(\ex^*(n,H)n^{k-3})$, 
    \item[(ii)] $\ex(n,H, \textup{rainbow-}H)=o(n^{k-1})$.
\end{itemize}
\end{proposition}

\begin{proof} Let $G$ be a properly colored, rainbow-$H$-free graph on $n$ vertices. To prove \textit{(i)}, let us pick an arbitrary edge of $G$; there are at most $\ex^*(n,H)$ ways to do this. Then we pick another edge of the same color; there are less than $n/2$ ways to do this. Next we pick $k-4$ additional vertices; there are $O(n^{k-4})$ ways to do this. Now note that there are at most $|V(H)|!$ copies of $H$ on each set of $|V(H)|$ vertices picked this way. This way we counted every copy of $H$ at least once, as every copy of $H$ must contain two edges of the same color. This proves \textit{(i)}.

Now let us prove \textit{(ii)}. Observe that \textit{(i)} immediately gives the upper bound $O(n^{k-1})$, as $\ex^*(n,H) = O(n^2)$. As $G$ contains $o(n^k)$ copies of $H$, by the graph removal lemma there is a set $E_0$ of $o(n^2)$ edges such that every copy of $H$ in $G$ contains an edge of $E_0$.

We know that every copy of $H$ in $G$ contains two edges of the same color, say $e_H$ and $e'_H$.

First let us count those copies of $H$ where the edge of $E_0$ in $H$ shares its color with another edge of $H$. In this case we can repeat the argument for \textit{(i)}. There are $o(n^2)$ ways to pick an edge $e$ of $E_0$ and less than $n/2$ ways to pick an edge $e'$ of the same color as $e$. The remaining $k-4$ vertices can be picked arbitrarily in $O(n^{k-4})$ ways and in any set of $|V(H)|$ vertices that we picked, we have at most $|V(H)|!$ copies of $H$.  So the total number of such copies of $H$ is $o(n^{k-1})$. 

Let us now count those copies of $H$ where the edge of $E_0$ in $H$ is disjoint from two edges $e_H, e'_H$ of $H$ having the same color. Then these three edges span six vertices. There are $o(n^2)$ ways to pick the edge of $E_0$, $O(n^2)$ ways to pick $e$, $O(n)$ ways to pick $e'$ (as it has the same color as $e$) and $O(n^{k-6})$ ways to pick the remaining vertices. So in total there are $o(n^{k-1})$ such copies of $H$.

Next we count those copies of $H$ where $uv$ is an edge of $E_0$ in $H$ and $e_H = uw$, $e'_H = xy$ are of the same color and these three edges span five different vertices. We can pick $uv$ in $o(n^2)$ ways, $xy$ in $O(n^2)$ ways, but then $uw$ can be picked in at most one way. The remaining vertices can be picked in $O(n^{k-5})$ ways, so the total number of such copies of $H$ is $o(n^{k-1})$.

Finally, we count those copies of $H$ where $uv$ is an edge of $E_0$ in $H$ and $e = uw$, $e' = vx$ are of the same color and these three edges span four different vertices. Then we can pick $uv$ in $o(n^2)$ ways, $uw$ in $O(n)$ ways, but then $vx$ can be picked in at most one way. The remaining $k-4$ vertices can then be picked in $O(n^{k-4})$ ways, so the total number of such copies of $H$ is again $o(n^{k-1})$. 

In all four cases we obtained $o(n^{k-1})$ copies of $H$, which finishes the proof of \textit{(ii)}.
\end{proof}

It would be interesting to determine if there is a graph $H$ for which the upper bound in Proposition \ref{GenBounds} \textit{(ii)} is sharp. In some special cases, it can be improved. Indeed, for example, we will see later that for $k$ odd, we have $\ex(n,C_k,\textup{rainbow-}C_k)=\Theta(n^{k-2})$. 


\section{Paths and Cycles}
\label{sec:PathsCycles}

We begin with a basic lemma that we will use in the proofs of Theorem \ref{path} and \ref{cycles}.

\begin{lemma}\label{rp} Let $U$ be a set of vertices, $A$ be a set of colors, $v_1,\dots,v_k$ be vertices such that for each $1\le i\le k-1$, $v_i$ and $v_{i+1}$ have at least $|U|+2|A|+5k-9$ common neighbors. Then there is a rainbow path $v_1u_1v_2u_2\dots v_{k-1}u_{k-1}v_k$ that does not use any color in $A$ and $u_i\not\in U$ for each $1\le i\le k$.

\end{lemma}

\begin{proof} We prove that for every $1\le j\le k$, there is a rainbow path $v_1u_1v_2u_2\dots v_{j-1}u_{j-1}v_j$ that does not use any color in $A$, $u_i\not\in U$ and $u_i\neq v_{i'}$ for any $1\le i\le j-1$ and $1\le i'\le k$. We use induction on $j$; the statement is trivial for $j=1$. Assume we could find a such a path $P(j)$. Then we want to find a vertex $x$ connected to both $v_j$ and $v_{j+1}$ such that $x\not\in U$, $x$ is not in $P(j)$, $x$ is not any $v_{i'}$ and the colors of $xv_j$ and $xv_{j+1}$ are not on the edges of $P(j)$, nor in $A$. The number of forbidden vertices is at most $|U|+2k-2$ (including $v_j$ and $v_{j+1}$). The number of forbidden colors is at most $|A|+2k-4$. Each of those colors is on at most one edge incident to $v_j$ and at most one edge incident to $v_{j+1}$, thus it forbids at most two additional vertices. Thus there are at most $|U|+2|A|+5k-10$ forbidden vertices, hence we can find a common neighbor of $v_j$ and $v_{j+1}$ that is not forbidden.
\end{proof}

Now we are ready to prove Theorem \ref{path}, which we restate here for convenience.

\begin{rstthm1.1} If $k\ge 5$, then
$\ex(n,P_k,\textup{rainbow-}P_k) = \Theta(n^{\floor{\frac{k}{2}}})$.
\end{rstthm1.1}
\begin{proof}
For the lower bound, we have to construct a graph $G$ with $\Theta(n^{\floor{\frac{k}{2}}})$ copies of $P_k$ and give a proper edge coloring of $G$ such that it is rainbow-$P_k$-free. 

If $k$ is odd, let $U_1$, $U_2, \ldots, U_k$ be disjoint sets of vertices of $G$ defined as follows. Let $U_i = \{u_{i,1}, u_{i,2}, \ldots, u_{i,b}\}$ for $i = 1, 3, 4$ and $i = 2j+1$ for all $j \ge 3$ where $b = \Theta(n)$ is chosen as large as possible. For $i = 2, 5$ and $i = 2j$ for all $j \ge 3$, simply let $U_i = \{u_i\}$. 
If $k$ is even, let $U_1$, $U_2, \ldots, U_k$ be disjoint sets of vertices of $G$ defined as follows. Let $U_i = \{u_{i,1}, u_{i,2}, \ldots, u_{i,b}\}$ for $i = 1, 3$ and $i = 2j$ for all $j \ge 2$ where $b = \Theta(n)$ is chosen as large as possible. For $i = 2$ and $i = 2j+1$ for all $j \ge 2$, simply let $U_i = \{u_i\}$. 

In both cases, sets of linear size and sets of size 1 alternate, with one exception, where large sets follow each other and in the odd case we have another exception, where singletons follow each other. But, as we will see, the important part is between $u_2$ and $u_5$.

The edges of $G$ and the colors of the edges are defined as follows. Let every vertex of $U_1$ and $U_3$ be adjacent to $u_2$ such that the edge $u_2u_{3,i}$ gets color $i$ for each $1 \le i \le b$. Let $u_{3,i}$ be adjacent to $u_{4,i}$ for each $1 \le i \le b$ and let all the vertices of $U_4$ be adjacent to $u_5$ such that the edge $u_{4,i}u_5$ gets color $i$ for each $1 \le i \le b$. Moreover, for each $i$ with $i \ge 5$, let all of the vertices in $U_i$ be adjacent to $U_{i+1}$. We extend the coloring already given to an arbitrary proper coloring. It is easy to see that any copy of $P_k$ in $G$ contains the edges $u_2u_{3,i}$ and $u_{4,i}u_5$ for some $i$, which have the same color by definition. Thus $G$ is rainbow-$P_k$-free and it is easy to see that it contains $b^{\floor{\frac{k}{2}}} = \Theta(n^{\floor{\frac{k}{2}}})$ copies of $P_k$, as desired.


\vspace{2mm}

Now we prove the upper bound. Let us recall first that a rainbow-$P_k$-free graph $G$ has $O(n)$ edges.
A pair of vertices $u, v$ is called a \emph{thin pair} if they have at most $5k-4$ common neighbors and a \textit{fat pair} otherwise.

We claim that there is no $k$-path $v_1v_2\dots v_k$ where $v_{2i}$ and $v_{2i+2}$ forms a fat pair for every $i$. Indeed, such a $k$-path would imply the existence of a rainbow-$k$-path by Lemma \ref{rp}.

Thus all the $k$-paths $v_1v_2\dots v_k$ have the property that $v_{2i}$ and $v_{2i+2}$ form a thin pair for some $1\le i\le (k-2)/2$. There are $O(n)$ ways to choose each of the edges $v_1v_2,v_3v_4,\dots,v_{2i-1}v_{2i}$ and each of the edges $v_{2i+2}v_{2i+3},v_{2i+4}v_{2i+5}\dots,v_{2l}v_{2l+1}$, where $2l+1$ is $k$ or $k-1$ depending on whether $k$ is even or odd. When $k$ is even, there are $n$ further ways to choose $v_k$. Finally, there are at most $5k-4$ ways to choose $v_{2i+1}$. Altogether there are $O(n^{\floor{{k}/{2}}})$ $k$-paths where $v_{2i}$ and $v_{2i+2}$ form a thin pair. As there are constant many ways to choose $i$, this finishes the proof. 
\end{proof}


It is easy to see that $\ex(n,P_3,\textup{rainbow-}P_3)=\ex(n,P_2,\textup{rainbow-}P_2)=0$. We determine $\ex(n,P_4,\textup{rainbow-}P_4)$ exactly. In fact we will completely characterize the graphs that have a proper edge-coloring without a rainbow $P_4$.

\begin{proposition}\label{p4}
A graph $G$ has a proper edge-coloring without a rainbow $P_4$ if and only if each connected component of $G$ is either a star, or a path, or an even cycle, or has at most four vertices. Consequently, $\ex(n,P_4,\textup{rainbow-}P_4)=12\lfloor n/4\rfloor$.
\end{proposition}

\begin{proof}
First we give a proper edge-coloring of $G$ without a rainbow $P_4$ if its components are as listed. The paths and even cycles are colored with two colors, thus they cannot contain a rainbow $P_4$. Stars do not contain any $P_4$, so any proper edge coloring is good. Finally, note that the proper edge-coloring of $K_4$ with 3 colors does not contain a rainbow $P_4$, so any graph with at most four vertices also has the same property.

Next let $G$ be a graph with a proper coloring of its edges without a rainbow $P_4$, and let us study the components of $G$. If a component does not contain $P_4$, it is a star or a triangle. 
If there is a fourth edge spanned by these four vertices, we obtain either a $C_4$ or a triangle with a hanging edge. It is easy to check that no edge can go out to a fifth vertex in either case without creating a rainbow $P_4$. Therefore, we may that assume there is no fourth edge spanned by $a,b,c$ and $d$. Then it is again easy to see that $b$ and $c$ cannot be connected to a fifth vertex, without creating a rainbow $P_4$, so if there is another edge, it creates a $P_5$. Now suppose we discovered already a $P_k$ for $k\ge 5$. The only further edge that can go between vertices of the path without creating a rainbow $P_4$ has to go between the endpoints, and so it creates a cycle.  Finally, in the same way as before, any edge going to a new vertex from the path must extend the path, in which case we can continue with a $P_{k+1}$. 

Hence the component is either a path, or contains a cycle. It is easy to see that a proper coloring of an odd cycle contains a rainbow $P_4$, while an edge added to an even cycle of length at least 6 creates a rainbow $P_4$ in any proper coloring. This shows that if the component is not a path, it must be an even cycle, finishing the characterization.

To finish the proof, note that $C_k$ contains $k$ copies, $P_k$ contains $k-2$ copies, $K_4$ contains 12 copies, and a star contains no copies of $P_4$. Therefore; the number of copies of $P_4$ is maximized if we take $\lfloor n/4\rfloor$ disjoint copies of $K_4$ (and a few isolated vertices).
\end{proof}

\vspace{2mm}

Below we determine the order of magnitude for odd cycles. We restate Theorem \ref{cycles} for convenience. Note that obviously $\ex(n,C_3,\textup{rainbow-}C_3)=0$.

\begin{rstthm1.2}[odd cycles] If $k\ge 2$, then $\ex(n,C_{2k+1},\textup{rainbow-}C_{2k+1})=\Theta(n^{2k-1})$.
\end{rstthm1.2}

\begin{proof} For the lower bound, replace each vertex of a $C_{2k+1}$ with linearly many vertices. For each edge of the original cycle, we put every possible edge between the corresponding sets, except for two non-adjacent edges, where we put only a matching. Let us color all the edges in the two matchings by the same color. Then a rainbow subgraph completely avoids one of the two matchings, hence it is bipartite and therefore is not a rainbow-$C_{2k+1}$. The number of copies of $C_{2k+1}$ is clearly $\Omega(n^{2k-1})$. Indeed, let us pick a vertex from each of the classes, except one of the classes incident to the first matching and one of the classes incident to the second matching. There are $\Omega(n^{2k-1})$ ways to pick these vertices and there is a unique $C_{2k+1}$ containing them.

For the upper bound we proceed somewhat similarly to Theorem \ref{path}. Again, we use thin pairs: this time a pair $u,v$ is thin if they have at most $5k-8$ common neighbors. Let us consider a $(2k+1)$-cycle $C=v_1v_2\dots v_{2k+1}v_1$. 



First we claim that for any $i$, one of the pairs $(v_i,v_{i+2}),(v_{i+2}v_{i+4}),\dots, (v_{i+2k-2},v_{i+2k})$ is thin (where addition in the subscripts is modulo $2k+1$). Assume otherwise and without loss of generality let the assumption be false for $i=1$. Then by Lemma \ref{rp} we can build a rainbow path from $v_1$ to $v_{2k+1}$, with the additional property that the colors on the edges of this path are different from the color on $v_{2k+1}v_1$. This path together with the edge $v_{2k+1}v_1$ forms a rainbow $C_{2k+1}$; a contradiction.
Thus one of the pairs is thin; without loss of generality it is $v_1,v_3$. Now applying the above claim with $i=3$, we obtain another thin pair. 
If it is $v_{2k+1},v_2$, then we can apply the above claim again, with $i=2$, to find a third thin pair. Anyways, this way at the end we find 
two thin pairs $v_i,v_{i+2}$ and $v_j,v_{j+2}$ in $C$, such that their clockwise order is $v_i,v_{i+2},v_j,v_{j+2}$ (note that $i+2=j$ or $j+2=i$ is possible). 

Now consider the following two ways of counting $(2k+1)$-cycles.
Either pick two disjoint thin pairs in $O(n^4)$ ways and $2k-5$ other vertices in $O(n^{2k-5})$ ways, or pick two thin pairs sharing a vertex in $O(n^3)$ ways and $2k-4$ other vertices in $O(n^{2k-4})$ ways. Then order them in a $2k-1$-cycle so that vertices in a thin pair are adjacent. This can be done in constant many ways. Finally we have in both cases constant many choices to put a common neighbor between the vertices of the thin pairs.
Clearly every $(2k+1)$-cycle is counted at least once in one of the two ways and both cases give $O(n^{2k-1})$ copies of $C_{2k+1}$, finishing the proof.
\end{proof}

Let us continue with even cycles.
The following theorem was proved in \cite{GGMV2017}.

\begin{theorem}[Gerbner, Gy\H ori, Methuku, Vizer; \cite{GGMV2017}]\label{ggymv} If $k\neq \ell$, then $\ex(n,C_{2\ell},C_{2k})=\Theta(n^\ell)$.
\end{theorem}

We will prove a generalization of this theorem in the rainbow setting. We follow the proof of Theorem~\ref{ggymv} given in \cite{GGMV2017}, which is based on the proof of the so-called reduction lemma of Gy\H ori and Lemons \cite{GyL2012}.
In \cite{GGMV2017} the authors also prove a stronger asymptotic bound than that in Theorem~\ref{ggymv} for the case when $\ell=2$. Here we will only determine the order of magnitude in this case, which helps to avoid some difficulties. During the proof we establish some properties of graphs with a proper edge-coloring without a rainbow-$C_{2k}$, and use these properties to obtain the upper bound $O(n^\ell)$ on the number of copies of $C_{2\ell}$. 

\smallskip

Now we are ready to prove Theorem \ref{cycles} for even cycles. We restate it here for convenience.

\begin{rstthm1.2}[even cycles]\label{evencycles} If $k \ge 2$, then $$\Omega(n^{k-1})= \ex(n,C_{2k},\textup{rainbow-}C_{2k})=O(n^k).$$
Moreover, if $k\neq \ell$, then $$\ex(n,C_{2\ell},\textup{rainbow-}C_{2k})=\Theta(n^\ell).$$
\end{rstthm1.2}

\begin{proof} For $k \neq \ell$, the lower bound follows from Theorem \ref{ggymv}. The following construction provides the lower bound $\Omega(n^{k-1})$ for $\ex(n,C_{2k},\textup{rainbow-}C_{2k})$. We take a blow-up of a copy $v_1v_2\dots v_{2k}v_1$ of $C_{2k}$ where we replace each of the vertices $v_3,v_6,v_8,v_{10},\dots, v_{2k-2},v_{2k}$ with classes of size about $n/(k-1)$ so that the resulting graph has $n$ vertices. 
We color the edges $v_1v_2$ and $v_4v_5$ with the same color. It is easy to see that every copy of $C_{2k}$ contains those edges, thus it is not rainbow and there are $\Omega(n^{k-1})$ copies of $C_{2k}$ in this graph.

For the upper bound, first we consider the case $k=2$. Observe that every proper coloring of $K_{2,4}$ contains a rainbow $C_4$, hence $\ex(n,C_{2\ell},\textup{rainbow-}C_{4})\le \ex(n,C_{2\ell},K_{2,4})=O(n^\ell)$ by a theorem of Gerbner and Palmer \cite{GP2017}.

Let us assume now $k>2$ and start with the case $\ell=2$. Let $G$ be a graph on $n$ vertices and assume we are given a coloring of $G$ without a rainbow-$C_{2k}$. Let $f(u,v)$ denote the number of common neighbors of $u$ and $v$. We call a pair of vertices $(u,v)$ \textit{fat} if $f(u,v)\ge 6k$ and a copy of $C_4$ is called {\it fat} if both opposite pairs in it are fat. First observe that the number of non-fat copies of $C_4$ is $O(n^2)$, as there are at most $\binom{n}{2}$ non-fat pairs and each such pair is an opposite pair in at most $\binom{6k-1}{2}$ copies of $C_4$.

We claim that the number of fat copies of $C_4$ is $O(n^2)$. To see this, we go through the fat copies of $C_4$ one by one in an arbitrary order and pick an edge (from the four edges of the $C_4$); we always pick the edge which was picked the smallest number of times before (in case there is more than one such edge, then we pick one of them arbitrarily).
When this procedure ends, every edge $e$ has been picked a certain number of times. Let us denote this number by $m(e)$ and call it the \textit{multiplicity} of $e$. Observe that $\sum_{e \in E(G)} m(e)$ is equal to the number of fat copies of $C_4$ in $G$. We will show that $m(e) < 16k^3$ for each edge $e$, thus the number of fat copies of $C_4$ in $G$ is at most $16k^3 \abs{E(G)} = O(n^2)$.

		Let us assume to the contrary that there is an edge $e=ab$ with $m(e) \ge 16k^3$. In this case we will find a rainbow-$C_{2k}$ in $G$, which will lead to a contradiction that completes the proof. More precisely, we are going to prove the following statement: 
		
		\begin{claim}
		For every $2 \le t \le k$ there is a rainbow-$C_{2t}$ in $G$, that contains an edge $e_t$ with $$m(e_t)\ge 16(2k-t)^3.$$ 
		\end{claim}
		\begin{proof}
		We proceed by induction on $t$. For the base step $t=2$ let us take a fat copy $abcda$ of $C_4$ containing $e_2=e=ab$. If the $C_4$ is not rainbow, we use its fatness to find at least $6k-2$ other $2$-paths between $a$ and $c$. At most one of those can share a color with the edges $ab$ or $bc$, thus we can replace the subpath $cda$ with another subpath to obtain a rainbow copy of $C_4$ containing $e_2$.
        
        Let us assume now that we have found a rainbow cycle $C$ of length $2t$, $t\leq k-1$, containing an edge $e_t=uv$ with multiplicity at least $16(2k-t)^3$. Let us consider the last $16t^2$ times $e_t$ was picked. This way we find a set $\cF_t$ of $16t^2$ fat copies of $C_4$ each containing $e_t$ and containing only edges with multiplicity at least $16(2k-t-1)^3$.
        
         At most $2\binom{2t-2}{2}$ of the copies of $C_4$ in $\cF_t$ have the other two of their vertices (besides $u$ and $v$) in $C$, as we have to pick two other vertices from $C$ and there are at most two $4$-cycles containing the edge $uv$ and two other given vertices. 
%
         	Thus, there are more than $12t^2$ fat $4$-cycles in $\cF_t$ that contain a vertex not in $V(C)$, let $\cF_t'$ be their set. Let $A$ be the set of vertices not in $V(C)$ that are neighbors of $u$ in a $4$-cycle in $\cF_t'$, and $B$ be the set of vertices not in $V(C)$ that are neighbors of $v$ in a $4$-cycle in $\cF_t'$. We claim that $|A|$ or $|B|$ is at most $2t$.
         	Indeed, suppose the contrary and look at the neighbours of $u$ in the $4$-cycles in $\cF_t'$. On the one hand, there are at most $2t$ of them in $V(C)$. Each of them appears in at most $2t$ of the $4$-cycles in $\cF_t'$, as the fourth vertices of these cycles are all different and are in $B$.
         	On the other hand, by assumption there are again at most $2t$ neighbours of $u$ outside $V(C)$ (those in the set $A$). Each of them is in at most $2t$ $4$-cycles in $\cF_t'$ where the fourth vertex is in $V(C)$ and, similarly as before, in at most $2t$ $4$-cycles in $\cF_t'$ where the fourth vertex is outside $V(C)$. This shows that together there can be at most $2t\cdot 2t+2t\cdot (2t+2t)=12t^2$ cycles in $\cF_t'$, a contradiction.

         	Without loss of generality let $u$ be the vertex which  has more than $2t$ neighbors in $4$-cycles in $\cF_t$ that are not in $V(C)$. Among these more than $2t$ vertices, at least one of them, call it $y$, has that the color on $yu$ is not used in $C$. Also, recall that the multiplicity of $yu$ is at least $16(k-t-1)^3$.
            As the $4$-cycles are fat, there are at least $6k$ common neighbors of $v$ and $y$ and at least $4k$ of those are not in $V(C)$. There are less than $2k$ colors that are used in $C$ and on the edge $yu$ and each of those colors appears at most once on edges that connect $y$ and $v$ to the at least $4k$ selected common neighbors. Therefore, there is a common neighbor $x$ such that the colors of the edges $vx$ and $yx$ are neither in $C$, nor on the edge $yu$. Then we can replace the edge $uv$ in $C$ with the edges $uy$, $yx$, $xv$ to obtain a rainbow cycle of length $2t+2$, which contains an edge (namely $uy$) with multiplicity at least $16(k-t-1)^3$.
        \end{proof}

This finishes the proof of the case $\ell=2$. Now we consider the case when $\ell\geq 3$. Note that we have $\frac{1}{2}\sum_{a\neq b,\,a,b \in V(G)}\binom{f(a,b)}{2}=O(n^2)$ as the left-hand side counts the number of copies of $C_4$.

		\begin{claim}\label{klem} For every $a\in V(G)$ we have $\sum_{b \in V(G) \setminus \{a\}} f(a,b)\le cn$ for some $c=c(k)$.
			
		\end{claim}

		Note that the left-hand side of the above inequality is the number of paths of length $3$ starting at $a$.
        
        \begin{proof}
        Let $N_1(a)$ be the set of vertices adjacent to $a$ and $N_2(a)$ be the set of vertices at distance exactly $2$ from $a$.
			Let $E_1$ be the set of edges induced by $N_1(a)$ and $E_2$ be the set of edges $uv$ with  $u \in N_1(a)$ and $v \in N_2(a)$. Then clearly $\sum_{b\in V(G) \setminus \{a\}}f(a,b)=2|E_1|+|E_2|$.
			
			We claim that $E_1 \cup E_2$ does not contain a copy of the $12k$-ary tree with depth $4k$. Assume it does contain such a copy $T$ and let $x_1$ be the root of $T$. In what follows, we will construct a path $P$ on $4k$ vertices starting at $x_1$. Let the next vertex be an arbitrary child $x_2$ of $x_1$. At later points we always pick the next vertex $x_{i+1}$ to be a child of $x_i$ such that both the color of $x_ix_{i+1}$ and the color of $ax_{i+1}$ (if exists) are different from the colors of $ax_j$ for every $j\le i$ (if exists) and from the colors used on the path earlier. As there are at most $8k$ forbidden colors, there are at most $8k$ children of $x_i$ that we cannot pick because of $x_ix_{i+1}$ and at most $4k-1$ children of $x_i$ that we cannot pick because of $ax_{i+1}$ (as the color of $ax_j$ is automatically avoided here). Hence we always have a neighbor to pick and we can really obtain a desired path $P$ in this way. 
            Observe that this path, together with the edges connecting some of its vertices to $a$ is rainbow. If $x_i$ and $x_{i+2k-2}$ are both in $N_1(a)$ for some $i$, then they, together with the vertices of $P$ between them and $a$ form a rainbow $C_{2k}$, a contradiction which finishes the proof. If all the edges of $P$ are in $E_2$, then $x_1$ and $x_{2k-1}$ or $x_2$ and $x_{2k}$ are both in $N_1(a)$, and the previous case applies. Hence we may assume that the edge $x_ix_{i+1}$ is in $E_1$ for some $i$. Without loss of generality we also may assume $i\le 2k$. Then $x_{i+2k-2}$ and $x_{i+1+2k-2}$ both have to be in $N_2(a)$ (otherwise the earlier case applies), but then the edge between them is not in $E_1\cup E_2$; a contradiction.
        \end{proof}
        
        From now on we follow the proof from \cite{GGMV2017} more closely, as we have already dealt with the difficulties arising from forbidding only rainbow copies of $C_{2k}$.
        
        		Let us fix vertices $v_1, v_2, \dots, v_\ell$ and let $g(v_1,v_2, \ldots, v_\ell)$ be the number of copies of $C_{2l}$ in $G$ of the form $u_1v_1u_2v_2\cdots u_\ell v_\ell u_1$ for some vertices $u_1,u_2,\dots,u_\ell$. 
		 Clearly $g(v_1,v_2, \ldots, v_\ell) \le \prod_{j=1}^\ell f(v_j,v_{j+1})$ (where $v_{\ell+1}=v_1$ in the product).

		 If we add up $g(v_1,v_2, \ldots, v_\ell)$ for all possible $\ell$-tuples $v_1, v_2, \dots, v_\ell$ of $\ell$ distinct vertices in $V(G)$, we count every $C_{2\ell}$ exactly $4\ell$ times. Therefore, the number of copies of $C_{2\ell}$ is at most
		
		\begin{equation}\label{equa3}\frac{1}{4\ell}\sum_{(v_1, v_2, \dots, v_\ell)}\prod_{j=1}^\ell f(v_jv_{j+1})\le\frac{1}{4\ell}\sum_{(v_1, v_2, \dots, v_\ell)}\frac{f^2(v_1,v_2)+f^2(v_2,v_3)}{2}\prod_{j=3}^\ell f(v_jv_{j+1}).
		\end{equation}
		
	Fix two vertices $u,v\in V(G)$ and let us examine what factor $f^2(u,v)$ is multiplied with in \eqref{equa3}. It is easy to see that $f^2(u,v)$ appears in \eqref{equa3} whenever $u=v_1,v=v_2$ or $u=v_2,v=v_1$ or $u=v_2,v=v_3$ or $u=v_3,v=v_2$.
		Let us consider the case $u=v_1$ and $v=v_2$, the other three cases are similar and give only an extra constant factor of $4$.
		In this case $f^2(u,v)$ is multiplied with $$\frac{1}{8\ell}\left(\prod_{j=3}^{\ell-1} f(v_jv_{j+1})\right)f(v_\ell,u) = \frac{1}{8\ell}f(u,v_\ell) \prod_{j=3}^{\ell-1} f(v_jv_{j+1})$$ for all the choices of $(\ell-2)$-tuples $v_3, v_4, \dots, v_\ell$ of distinct vertices. We claim that $$\sum_{(v_3, v_4, \dots, v_\ell)}\frac{1}{8\ell}f(u,v_\ell)\prod_{j=3}^{\ell-1} f(v_jv_{j+1})\le \frac{c^{\ell-2}n^{\ell-2}}{8\ell},$$ where $c=c(k)$ is the constant from Claim \ref{klem}. Indeed, we can rewrite the left-hand side as 
		\[
		\frac{1}{8\ell}\left(\sum_{v_\ell\in V(G)}f(u,v_\ell)\left(\sum_{v_{\ell-1}\in V(G)}f(v_\ell,v_{\ell-1})\cdots \left(\sum_{v_{4}\in V(G)}f(v_5,v_4)\left(\sum_{v_{3}\in V(G)}f(v_4,v_{3})\right)\right)\cdots\right)\right).
		\] After repeatedly applying Claim \ref{klem} we arrive at the desired upper bound and this finishes the proof.
\end{proof}

We do not know the order of magnitude of $\ex(n,C_{2k},\textup{rainbow-}C_{2k})$ already for $k = 2, 3$, but we can improve Theorem \ref{cycles} slightly in these cases.

\begin{proposition} 
\label{c4}
We have $\ex(n,C_4,\textup{rainbow-}C_4)=\Omega(n^{3/2})$.
\end{proposition}

\begin{proof}
Let us take two isomorphic $C_4$-free graphs $G$ and $G'$ on $\lfloor n/2\rfloor$ vertices each with $\Omega(n^{3/2})$ edges. Connect every vertex $v$ in $G$ to its copy $v'$ in $G'$. Denote the resulting graph by $G_0$. Let us color these edges with color 1 and extend this to an arbitrary proper coloring of the edges of $G_0$. Every copy of $C_4$ in $G_0$ has to use vertices from both $G$ and $G'$, thus some edge $vv'$ of color 1. The neighbor of $v$ in the $C_4$ must be a vertex in $G$ and the neighbor of $v'$ must be a vertex in $G'$. They can only be connected by another edge of color 1, thus the $C_4$ is not rainbow. On the other hand, for every edge $uv$ of $G$ the $4$-cycle $uvv'u'u$ is in $G_0$, thus there are $\Omega(n^{3/2})$ copies of $C_4$ in $G_0$.
\end{proof}

\begin{proposition}
\label{c6}
We have  $\ex(n,C_6,\textup{rainbow-}C_6)=O(n^{8/3})$.

\end{proposition}

\begin{proof} Let $G$ be a rainbow-$C_6$-free graph and $v_1v_2v_3v_4v_5v_6v_1$ be a 6-cycle in it. Note that $G$ has $O(n^{4/3})$ edges. We call a pair of vertices fat if they have at least $11$ common neighbors. If both pairs $(v_1,v_3)$ and $(v_3,v_5)$ are fat, then we can find a rainbow $C_6$ of the form $v_1uv_3u'v_5v_6v_1$. Indeed, we can apply Lemma \ref{rp} for $v_1,v_3,v_5$ with $U=\{v_6\}$ and $A$ containing the color of $v_5v_6$ and $v_6v_1$. This way we find a rainbow path $v_1uv_3u'v_5$ avoiding $v_6$ and the colors in $A$, thus it forms a rainbow 6-cycle with the edges $v_5v_6$ and $v_6v_1$. As a consequence we have that there are at most two fat pairs among the pairs $(v_i,v_{i+2})$ (where addition in the indices is modulo $6$).

Let us count first the 6-cycles with two fat pairs. Observe that those pairs cannot share a vertex, thus there are only two possible configurations: either one of the pairs has a vertex between the two vertices of the other pair (like $(v_1,v_3)$ and $(v_2,v_4)$) or not (like $(v_1,v_3)$ and $(v_4,v_6)$). To count those $6$-cycles where $(v_1,v_3)$ and $(v_4,v_6)$ are the fat pairs we can pick the edges $v_1v_6$ and $v_3v_4$ in $O(n^{8/3})$ ways and there are constant many ways to pick $v_2$ connected to both $v_1$ and $v_3$ and $v_5$ connected to both $v_4$ and $v_6$.
For those $6$-cycles where $(v_1,v_3)$ and $(v_2,v_4)$ are the fat pairs we pick the edges $v_1v_6$ and $v_3v_4$ in $O(n^{8/3})$ ways and similar to the previous case, there are constant many ways to pick the remaining vertices.

To count $6$-cycles where at most one pair, say $v_1v_3$, is fat, we pick the edges $v_1v_2$ and $v_4v_5$ and proceed similarly as above.
\end{proof}

\section{Trees and Forests}
\label{sec:TreesForests}

The aim of this section is to prove Theorem \ref{trees} and some additional results about forests. Let us first prove the following proposition that is used later in this section.

\begin{proposition}\label{disco}
Let $H$ be a graph that is neither a star, nor a triangle and let $c$ be the number of connected components of $H$. Then $\ex(n,H,\textup{rainbow-}H)=\Omega(n^c)$.
\end{proposition}

\begin{proof} Let $m$ be the largest chromatic number of a component of $H$. Let us consider a graph $G$ that contains linearly many vertex-disjoint copies of each component. Then $G$ obviously contains $\Omega(n^c)$ copies of $H$. For each component $H'$ of $H$, we color each copy of it the same way: using colors from 1 to $\chi(H')$. This way we properly color $G$ with $m$ colors. Since $H$ has more than $m$ edges, this implies that there is no rainbow copy of $H$ in $G$, finishing the proof.
\end{proof}

Next we will determine the order of magnitude for double stars. 

\begin{proposition}\label{doublestar} 
If $p,r\geq 1$, then $\ex(n,S_{p,r},\textup{rainbow-}S_{p,r})=\Theta(n)$.
\end{proposition}

\begin{proof} The lower bound on $\ex(n,S_{p,r},\textup{rainbow-}S_{p,r})$ follows from Proposition \ref{disco}. 

For the upper bound, assume without loss of generality that $p\le r$ and consider a properly edge-colored graph $G$ on $n$ vertices without a rainbow copy of $S_{p,r}$. We want to bound the number of copies of $S_{p,r}$ in $G$.

We claim that if a vertex $v$ has degree more than $2p+r$ in $G$, than it cannot be a center of a copy of $S_{p,r}$. Indeed, if $v$ is a center of some copy, then it has a neighbor $u$ which has at least $p$ neighbors different from $v$. 
Let us choose a set $A$ of size $p$ out of these neighbors of $u$ arbitrarily. Then $v$ has at least $p+r$ neighbors not in $A$ and different from $u$ and at least $r$ of them do not have any of the colors that appear on the edges between $u$ and vertices in $A$. Thus, those $r$ vertices together with $u$, $v$ and $A$ form a rainbow copy of $S_{p,r}$, which contradicts our assumption.

Now we are ready to count the copies of $S_{p,r}$ in $G$. We can pick a center of it in at most $n$ ways. Then we can pick one of its neighbors to be the other center in at most $2p+r$ ways and there are at most $\binom{2p+r-1}{r}$ and $\binom{2p+r-1}{p}$ ways to pick the $p$ and $r$ other neighbors of the centers, respectively. Thus, there are $O(n)$ copies of $S_{p,r}$ in $G$ and this finishes the proof.
\end{proof}

Let us now turn to the proof of Theorem \ref{trees}. We restate it here for convenience.

\begin{rstthm1.3}
If $T$ is a tree with $t$ vertices that is neither a star nor a double star, then
\begin{equation*}
    \ex(n,T,\textup{rainbow-}T)=\Omega(n^{\left\lceil t/4\right \rceil}).
\end{equation*}
\end{rstthm1.3}

\begin{proof}

First we handle the case $T=P_5$ separately. Theorem \ref{path} gives $\ex(n,P_5,\textup{rainbow-}P_5)=\Omega(n^2)$, from which the statement follows for $P_5$. Therefore, from now on we may assume that $T$ is not $P_5$. 

Let $L$ be the set of leaves in $T$, $\ell=|L|$ and $T'$ be the tree we obtain by deleting the vertices in $L$ from $T$. To begin with we will establish separate bounds in different cases.

\smallskip

Suppose first that there are two independent edges in $T'$. Replace each leaf in $T$ by linearly many copies of it to obtain $H$. Then $H$ has $\Omega(n^{\ell})$ 
copies of $T$. Now consider any proper edge-coloring of $H$ where the two independent edges of $T'$ receive the same color. Then any rainbow subgraph of $H$ has to avoid at least one of them, but then it must have fewer non-leaf edges (edges that are not incident to leaves) than $T$. This shows that $H$ is rainbow-$T$-free.

\smallskip

If $T'$ does not have two independent edges, then $T'$ is a star. Denote its center by $u$. As $T$ is not a star, nor a double star, $T'$ must have at least two leaves. Let $v$ be a leaf in $T'$ that has the smallest degree in $T$ and let $w$ be another leaf in $T'$. Next replace each leaf of $T$ not adjacent to $v$ by linearly many copies of it to obtain the graph $H'$. The number of copies of $T$ in $H'$ is $\Omega(n^m)$, where $m$ is the number of leaves in $T$ not adjacent to $v$. Note that we have $m\ge 2$ whenever $T$ is not $P_5$.  Consider now any proper edge-coloring of $H'$ where the edge $uw$ has the same color as the edge $vv'$ for some leaf neighbor $v'$ of $v$. Then any connected rainbow subgraph $F$ of $H'$ has to avoid at least one of the edges $uw$ and $vv'$. If $uw$ is not in $F$, then, as $H'$ is a tree, necessarily $u$ or $w$ is not in $F$, thus $F$ has fewer non-leaf vertices than $T$. If $vv'$ is not in $F$, then $F$ may have the same number of non-leaf vertices as $T$, but one of them will have smaller degree in $F$ than any of the non-leaf vertices in $T$. This shows that $H'$ is rainbow-$T$-free.

\smallskip

Next assume $T$ contains a \textit{bare} path $v_1v_2v_3v_4$, i.e., a path such that $v_2$ and $v_3$ have degree $2$ in $T$. Without loss of generality we may assume that the degree of $v_1$ is not $2$. Replace $v_2$ with vertices $u_1,\dots,u_b$ and $v_3$ with vertices $u_1',\dots, u_b'$ for some $b=\Theta(n)$ and connect $v_1$ to $u_i$, $u_i$ to $u_i'$ and $u_i'$ to $v_4$ for every $1\le i\le b$ to obtain the graph $G$. Then replace every vertex of degree $1$ in $G$ again by $b=\Theta(n)$ copies of itself to obtain the graph $G'$. Finally, consider the vertices of degree 2 in $T$ different from $v_2$ and $v_3$ whose degree in $G'$ is still $2$. We go through them in an order such that every vertex has at most one of its neighbors before it\footnote{It is a well-known fact that for forests such an ordering exists.} and we replace each of these vertices of degree 2 again with $b=\Theta(n)$ copies of themselves. In case the degree of some vertex $v$ under consideration becomes larger than two before we would arrive at it (because $v$ had a neighbor before it, and thus now has $\Theta(n)$ neighbors), we simply skip it. Let $G''$ denote the graph obtained this way and  consider any proper edge-coloring of $G''$ in which the edges $v_1u_i$ and $u_i'v_4$ have color $i$ for every $1\le i\le b$ and the edges $u_1u_1',u_2u_2',\dots,u_bu_b'$ have all the same color $b+1$. 

Observe that if $w\neq v_2,v_3$ has degree 2 in $T$ and when passing from $G'$ to $G''$ it is replaced with $w_1,\dots, w_b$, then in any subtree of $G''$, at most one of the $w_i$s can have degree $2$, as otherwise there would be a cycle of length $4$. Now let $Q$ be a copy of $T$ in $G''$. Clearly, every vertex of degree at least 2 in $Q$ is either a vertex of degree at least 2 in the original copy of $T$, or one of the vertices replacing it. We now distinguish two cases. First suppose that $Q$ contains at most one of the vertices $v_1$, $v_4$. Then, to be able to accommodate all the vertices of degree at least $2$ from $T$, it has to contain at least two  of the vertices $u_1,\dots,u_b,u_1',\dots,u_b'$ as vertices of degree $2$ in $T$. However, then it also has to contain at least two of the edges $u_1u_1',u_2u_2',\dots,u_bu_b'$, all of which have the same color. If $Q$ contains both vertices $v_1$, $v_4$, then, as $Q$ is connected, there has to be a path from $v_1$ and $v_4$, which has to be of the form $v_1u_iu_i'v_4$ for some $i$ and hence there are again two edges in $Q$ of the same color. In any case, $Q$ is not rainbow, hence $G''$ is rainbow $T$-free.

Let us now count the number of `canonical' copies of $T$ in $G''$: those which contain for every vertex $v$ of the original copy of $T$ either $v$, or one of the $\Theta(n)$ vertices it was replaced with. First of all, we have $\Omega(n)$ choices for the path $v_1v_2v_3v_4$. Then, as in $G'$ there are at least $\ell-1$ vertices of degree $1$ (we might have lost one when creating $G$) and each of them was replaced by linearly many copies when passing to $G'$, to choose the vertices for the leaves we have $\Omega(n^{\ell-1})$ options. Now, as $T$ has $\ell$ leaves, it has at most $\ell-2$ vertices of degree greater than 2 and hence at least $t-2\ell+2$ vertices of degree $2$. By the time we create $G'$, we can `lose' at most $\ell+2$ of these vertices, but we will still be left with at least $t-3\ell$ vertices of degree $2$ in $G'$. Let $T''$ be obtained from $T$ by contracting edges alongside vertices of degree 2, i.e., replacing every maximal bare path by a single edge. Then $T''$ has at most $2\ell-2$ vertices and hence at most $2\ell-3$ edges. Each such edge represents a maximal bare path in $T$ and, in particular, this means that $G'$ can contain also only at most $2\ell-3$ such maximal bare paths. Note that here we used our assumption on the degree of $v_1$, which ensures that the path $v_1v_2v_3v_4$ is at the end of a maximal bare path. Then, those at least $t-3\ell$ vertices which have degree $2$ both in $T$ and in $G'$ are divided into at most $2\ell-3$ paths. If such a path contains $i$ vertices of degree $2$, then (by possibly modifying the order in which the vertices are handled) when passing from $G'$ to $G''$ at least $\lceil i/2\rceil$ of them are going to be replaced by linearly many vertices in $G''$. Each such replacement gives us $\Omega(n)$ choices to pick the corresponding vertex, which together leaves us with $\Omega(n^{\left\lceil\frac{t-3\ell}{2}\right\rceil})$ choices. Summing up we get that the number of `canonical' copies of $T$ in $G''$ is $\Omega(n^{\left\lceil\frac{t-\ell}{2}\right\rceil})$. This bound is good when $\ell$ is small with respect to $t$. Otherwise, we can skip the last round above and only consider the choices for our bare path and the leaves. In this way we get $\Omega(n^\ell)$ `canonical' copies of $T$ in $G''$. These two bounds together give a lower bound of $\Omega(n^{\left\lceil t/3\right\rceil})$.

\smallskip

Now assume $T$ contains no bare path $v_1v_2v_3v_4$. We claim that in this case we have $t\leq 4\ell-3$. As earlier observed, if $T$ has $\ell$ leaves, then it has at most $\ell-2$ vertices of degree larger than $2$. Hence, to prove the desired inequality, it is enough to show that the number of vertices of degree $2$ is at most $2\ell-3$. In the previous paragraph we showed that $T$ can contain at most $2\ell-3$ maximal bare paths and, by our new assumption, a maximal bare path can contain at most one vertex of degree $2$. Therefore, the number of vertices of degree $2$ is indeed at most $2\ell-3$.

\smallskip

Now we are ready to put together the different bounds obtained so far. If $t>4\ell-3$, then by the above reasoning $T$ must contain a bare path on four vertices and hence $\ex(n,T,\textup{rainbow-}T)=\Omega(n^{\left\lceil t/3\right\rceil})$ as given by $G''$. If $t\le 4\ell-3$ but $T'$ is not a star, we have $\ex(n,T,\textup{rainbow-}T)=\Omega(n^\ell)=\Omega(n^{\left\lceil t/4 \right\rceil})$ as given by $H$. Finally, if $T'$ is a star we will consider $H'$. Note that by assumption $t\le 2\ell+1$ and recall that in $H'$ the vertex $v$ is a leaf of $T'$ that has the smallest degree in $T$ and $m$ is the number of leaves in $T$ that are not adjacent to $v$. Then we have $m\geq \left\lceil \ell/2\right\rceil$, with equality only if $T'$ has two leaves, in which case $t=\ell+3$.
If $m>\left\lceil l/2 \right\rceil$, or $m=\left\lceil l/2 \right\rceil$ and $t\le 2\ell$, then we have $m\geq \left\lceil t/4 \right\rceil$,
and hence  $\ex(n,T,\textup{rainbow-}T)=\Omega(n^m)=\Omega(n^{\left\lceil t/4 \right\rceil})$ as shown by $H'$. In the remaining case $m=\left\lceil l/2 \right\rceil$ and $t=2l+1$ we have $T=P_5$, which we have already dealt with. This completes the proof.
\end{proof}

Let us note that 
Theorem \ref{trees} can likely be improved. We remark that if $T$ contains two adjacent vertices of degree 2, our proof gives the lower bound $\Omega(n^{t/3})$. However, our main goal was to show that $\ex(n,T,\textup{rainbow-}T)$ grows fast with the number of vertices of $T$.

\smallskip

In the remainder of this section we prove some sporadic results about special forests.

\begin{proposition}\label{twostar} 
If $F$ is a forest consisting of two stars, then
\begin{equation*}
\ex(n,F,\textup{rainbow-}F)=\Theta(n^2).   
\end{equation*}
\end{proposition}

\begin{proof} The lower bound on $\ex(n,F,\textup{rainbow-}F)$ follows from Proposition \ref{disco}. 

For the upper bound let us denote the two stars in $F$ by $S_p$ and $S_r$, $p\le r$ and let $G$ be a properly edge-colored graph on $n$ vertices without a rainbow copy of $F$. 

Suppose first that $G$ contains a vertex $v$ of degree more than $2p+r$. We claim that this vertex $v$ has to be in every copy of $S_p$ in $G$. Indeed, assume to the contrary that $S$ is a copy of $S_p$ not containing $v$. Now at most $p+1$ neighbors of $v$ are in $S$ and at most $p$ vertices are connected to $v$ using a color from $S$. Thus we can find $r$ neighbors of $v$, which together with $v$ form a copy of $S_r$ and with $S$ this gives a rainbow copy of $F$ in $G$; a contradiction. So $v$ is contained in every copy of $S_p$. However $F$ itself contains two disjoint copies of $S_p$ which implies that $G$ is actually $F$-free.

Therefore, we may assume that every vertex in $G$ has degree at most $2p+r$. Then, when counting the number of copies of $F$ in $G$, there are at most $\binom{n}{2}$ ways to choose the two centers for the stars and at most $\binom{2p+r}{p}\binom{2p+r}{r}$ ways to choose the leaves afterwards. Together this shows that the number of copies of $F$ is indeed $O(n^2)$.
\end{proof}

We remark that if $F$ is made up of more stars we cannot hope for a bound that depends only on the number of stars. To see this, let $F$ consist of three components, two of which are single edges and one that is a star $S_r$, $r\geq 1$. Let $G$ be a graph with three components, two of which are also only single edges and one that is a star $S_{n-5}$. Consider a proper edge-coloring of $G$ where the two edge components have the same color. Then $G$ contains $\binom{n-5}{r}$ copies of $F$, but no rainbow copy. However, if all the components of $F$ are single edges, we can obtain the following.

\begin{proposition} 
Let $M_k$ be a matching with $k>1$ edges. Then 
\begin{equation*}
\ex(n,M_k,\textup{rainbow-}M_k)=\Theta(n^k).
\end{equation*}
\end{proposition}

\begin{proof} The lower bound on $\ex(n,M_k,\textup{rainbow-}M_k)$ follows from Proposition \ref{disco}. 

For the upper bound consider a properly edge-colored graph $G$ on $n$ vertices without a rainbow copy of $M_k$. Then according to \cite[Theorem 1]{jps} $G$ has $O(n)$ edges. To find a copy of $M_k$ we have to pick $k$ edges which can be done in $O(n^k)$ ways.
\end{proof}

\section{Concluding remarks and Open problems}
\label{sec:ConcludingOpen}

In this paper we determined the order of magnitude of $\ex(n,F,\textup{rainbow-}F)$ for paths and odd cycles and obtained bounds for even cycles and trees. Several interesting questions remain open. Below we mention a few of them.

\begin{itemize}
\item Let $K_r$ denote a ciique on $r$ vertices. A natural question is to determine the order of magnitude of $\ex(n,K_r,\textup{rainbow-}K_r)$ for $r \ge 4$. Proposition \ref{GenBounds}, part (ii) implies that $\ex(n,K_r,\textup{rainbow-}K_r) = o(n^{r-1})$. It is easy to see that $\ex(n,K_r,\textup{rainbow-}K_r) = \Omega(n^{r-2})$. Indeed, partition the $n$ vertices into $r$ parts $S_1, S_2, \ldots, S_r$ of roughly the same size, and take a matching $M_1$ between the parts $S_1$ and $S_2$ and take another matching $M_2$ between the parts $S_3$ and $S_4$ and the edges of both $M_1$ and $M_2$ are colored with the same color. Between every other pair of parts take a complete bipartite graph. It is easy to check that there are $\Omega(n^{r-2})$ copies of $K_r$ in this graph and in any copy of $K_r$ we must have an edge of $M_1$ and an edge of $M_2$, both of which are colored the same. So there is no rainbow copy of $K_r$.

\item What is the order of magnitude of $\ex(n,C_{2k},\textup{rainbow-}C_{2k})$? Theorem \ref{evencycles} proves some bounds on this function. Proposition \ref{c4} and Proposition \ref{c6} provide improved bounds in the case when $k = 2, 3$.

\item Theorem \ref{trees} shows that when $T$ is a tree, then with the exception of stars and double stars, $\ex(n,T,\textup{rainbow-}T)$ grows fast with the number of vertices if $T$ is a tree. We suspect that a similar phenomenon might be true for general graphs, with some small set of exceptions. One such exception we have encountered was the disjoint union of two stars. Another example is $T_p$, the triangle with $p$ leaves attached to one of its vertices. For this graph a simple case analysis shows that $\ex(n,T_p,\textup{rainbow-}T_p)=O(n)$.

\item In this paper we introduced the function $\ex(n,H,\textup{rainbow-}F)$ and studied it when $H = F$. It would be interesting to study the case when $H$ and $F$ are different graphs. 
\end{itemize}

\subsection*{Acknowledgements}

We thank Clara Shikhelman for proposing the problem studied in this paper at the Novi Sad Workshop on Foundations of Computer Science (NSFOCS) in July 2017. We are grateful to the organizers of the workshop for their hospitality and to Chris Dowden, Clara Shikhelman and Tuan Tran for fruitful discussions on the topic.

\end{document}